\documentclass[a4paper,11pt]{article}
\usepackage[english]{babel}
\usepackage[T1]{fontenc}
\usepackage[latin1]{inputenc}
\usepackage{amsmath,amssymb,mathrsfs}
\usepackage{soul}
\usepackage{xypic}
\usepackage{amsthm}
\usepackage[left=2cm,right=2cm,top=3cm,bottom=2cm]{geometry}
\theoremstyle{plain}
\newtheorem{thm}{Theorem}[section]
\newtheorem{prop}[thm]{Proposition}
\usepackage{hyperref}

\theoremstyle{definition}
\newtheorem{defi}[thm]{Definition}
\newtheorem{rem}[thm]{Remark}

\newtheorem{lem}[thm]{Lemma}

\title{Application of equivalence method to Monge-Ampère equations \textit{Elliptic case}}
\author{Imsatfia Moheddine\footnote{Institute Mathematics of Jussieu- Paris, UMR 7586,
Bâtiment Sophie Germain, Office 751, Case 7012 - 75205  Paris Cedex 13, France.}}
\usepackage{hyperref}
\date{}

\begin{document}

\maketitle

\begin{abstract}
The application of equivalence method to classify Monge-Ampère system leads to three orbits, parabolic case, hyperbolic case and elliptic case wich correspond to three types of Monge-Ampère systems.  In this paper we will study the elliptic case and give a presentation of the group as a complex group.
\end{abstract}

 
 \begin{center}
\begin{Large}
\section{Introduction}
\end{Large}
\end{center}

Cartan's method to state the equivalence problem devolopped by Elie Cartan in the years 1905-1910 more recently in \cite{Neut}, is a crucial tool. We are here interested in its application to the study of Monge-Ampère equation in 2 variables. Hence following the work of R. Bryant, D. Grossman and P. Griffiths in the years 1997-1998 \cite{BRP} in order to clarify the strategy of Cartan: Given a 5-dimensional contact manifold $(\mathcal{M},I)$ they applied this method to
some Monge-Ampère equations given in \cite{Motimoto79} $\varepsilon=\{\theta,d\theta,\Psi\}$ for $\theta\in\Gamma(I)$ and $\Psi$ is a 2-form. We may assume $d\theta\wedge\Psi=0$ mod $\{I\}$. On the
contact manifold $\mathcal{M}$, one can locally find a coframing $\eta=(\eta^a)$ such that $\eta^0\in\Gamma(I)$ and
\begin{equation}\label{equa1}
 d\eta^0=\eta^1\wedge\eta^2+\eta^3+\eta^4 \text{mod }\{I\},
\end{equation}
then we can write $\Psi=\frac{1}{2}b_{\imath\jmath}\eta^\imath\wedge\eta^\jmath$. We can find that there are three types of Monge-Ampère systems: 
\begin{enumerate}
 \item If $\Psi\wedge\Psi$ is a negative multiple of $d\eta^0\wedge d\eta^0$, then the local coframing $\eta$ my be chosen so that in addition to (\ref{equa1}),
\[
 \Psi=\eta^1\wedge\eta^2-\eta^3+\eta^4 \text{mod }\{I\};
\]
for a classical problem, this occurs when the Euler-Lagrange PDE is hyperbolic.
\item If $\Psi\wedge\Psi=0$, then the local coframing $\eta$ my be chosen so that in addition to (\ref{equa1}),
\[
 \Psi=\eta^1\wedge\eta^3 \text{mod }\{I\};
\]
for a classical problem, this occurs when the Euler-Lagrange PDE is parabolic.
\item If $\Psi\wedge\Psi$ is a positive multiple of $d\eta^0\wedge d\eta^0$, then the local coframing $\eta$ my be chosen so that in addition to (\ref{equa1}),
\[
 \Psi=\eta^1\wedge\eta^4-\eta^3+\eta^2 \text{mod }\{I\},
\]
for a classical problem, this occurs when the Euler-Lagrange PDE is elliptic.
\end{enumerate}

 R. Bryant, D. Grossman and P. Griffiths applied the equivalence method \cite{Olver} to study the hyperbolic case. Our aim here is to study the elliptic case in which we apply the equivalence method, using a presentation of the acting group as a complex group. We determined the case when an elliptic Monge-Ampère system is locally equivalent to the Monge-Ampère system for the linear homogeneous Laplace equations and the case when it is locally equivalent to an Euler-Lagrange system.

\section{Monge-Ampère System}

Denote by $\mathcal{J}^1(\mathbb{R}^2,\mathbb{R})$ be the first order jet espace $$\mathcal{J}^1(\mathbb{R}^2,\mathbb{R}):= \{(x^1,x^2,z,p_1,p_2)\in\mathbb{R}^2\times
\mathbb{R}\times\mathbb{R}^2\}.$$ 
For all smooth function $u:\mathbb{R}^2\rightarrow\mathbb{R}$, we associate the graph of $u$ by $\Sigma:= j^1u(\mathbb{R}^2)\subset\mathcal{M}$ with 
\[
 j^1u:\mathbb{R}^2\rightarrow\mathcal{J}^1(\mathbb{R}^2,\mathbb{R}).
\]
We have $z\circ(\mathcal{J}^1(u))=u$ and for $1\leq a\leq 2$, $p_a\circ(\mathcal {J}^1(u))=\frac{\partial u} {\partial x^a}$, then
\[
 \Sigma:=\left\{\left(x^1,x^2,u(x),\frac{\partial u}{\partial x^1},\frac{\partial u}{\partial x^2}\right), x\in\mathbb{R}^2\right\},
\]
is a smooth submanifold of $\mathcal{J}^1(\mathbb{R}^2,\mathbb{R})$. We define the one-form $\theta$ which is not closed by
\[
 \theta=dz-p_1dx^1-p_2dx^2.
\]
If $\Psi$ is a two form over $\mathcal{J}^1(\mathbb{R}^2,\mathbb{R})$, then
\[
 \Psi=\Psi_{p_1p_2}dp_1\wedge dp_2+\Psi_{p_1x^2}dp_1\wedge dx^2+\Psi_{p_2x^2}dp_2\wedge dx^2+\Psi_{p_1x^1}dp_1\wedge dx^1
\]
\[
+\Psi_{p_2x^1}dp_2\wedge dx^1+\Psi_{x^1x^2}dx^1\wedge dx^2\ \ \ \text{mod}(\theta).\ \ \ \ \ \ \ \ \ \ \ \ \ \ \ \ \ \ \ \ \ \ \ \ \ \ \
\]
A Monge-Ampère equation reads
\[
 \left\{
\begin{array}{c}
 \Psi\vert_\Sigma=0,\ \ \ \ \ \ \ \ \ \ \ \ \ \ \ \ \ \ \ \\
\theta\vert_\Sigma=0\ \ (\Rightarrow d\theta\vert_\Sigma=0),\\
dx^1\wedge dx^2\vert_\Sigma\neq0,\ \ \ \ \ \ \ \ \
\end{array}
\right.
\]
Along of $\Sigma$, we have
\[
 \Psi_{p_1p_2}\left[\frac{\partial ^2u}{(\partial x^1)^2}\frac{\partial ^2u}{(\partial x^2)^2}-\left(\frac{\partial ^2u}{\partial x^1\partial x^2}\right)^2\right]+\Psi_{p_1x^2} \frac{\partial ^2u}{(\partial x^1)^2}+\Psi_{p_2x^1}\frac{\partial ^2u}{(\partial x^2)^2}
\]
\[
+(\Psi_{p_1x^1}+\Psi_{p_2x^2})\frac{\partial ^2u}{\partial x^1\partial x^2}+\Psi_{x^1x^2}\left(x^a,u(x),\frac{\partial u}{\partial x^a}\right)=0.
\]
Without loss of generality we can normalize by assuming $\Psi$, we have
\[
 \Psi_{p_1x^1}=\Psi_{p_2x^2}\Leftrightarrow d\theta\wedge \Psi=0\ \ \text{mod}(\theta),
\]
Hence the data of Mong-Ampère equation are :
\[
 \left\{
\begin{array}{c}
 \theta\in\Omega^1(\mathcal{M})\text{ and } \Psi\in\Omega^2(\mathcal{M}),\\
\Psi\wedge d\theta=0\text{ mod}(\theta),\ \ \ \ \ \ \ \ \ \ \ \\
\theta\wedge d\theta\wedge d\theta\neq0,\ \ \ \ \ \ \ \ \ \ \ \ \ \ \
\end{array}
\right.
\]

\begin{lem}\label{lem2}

 Let $\varLambda=L(x,z,p)dx$ be a 1-form, for $x\in\mathbb{R}^n$, $z=u(x)$ and $p=(p_\imath)=\frac{\partial u}{\partial x^\imath}$. Over $\mathcal{J}^1(\mathbb{R}^n,\mathbb{R})$; we introduce the contact form, $\theta=dz-p_\imath dx^\imath$. If $\Sigma\subset\mathcal{J}^1 (\mathbb{R}^n,\mathbb{R})$ define by $\Sigma=j^1u(\Omega)=\{(x,u(x),p);\ \ x\in\Omega\subset\mathbb{R}^n\}$ with,
\[
\left\{ 
 \begin{array}{c}
  dx\vert_{\Sigma}\neq 0,\\
\theta\vert_{\Sigma}\neq 0,
 \end{array}
\right.
\]
 Then there exists a unique form $\Xi$ and a 1-form $\alpha$ as
\[
 d\varLambda=\theta\wedge\Xi+d\alpha,
\]
\end{lem}
\begin{rem}
With some additional conditions, theses variationals problems become a Euler-Lagrange equation of type Monge-Ampère. 
\begin{enumerate}
 \item The lemma \ref{lem2} is true for any form $\Lambda\in\Omega^n(\mathcal{M})$.
 \item Along $\Sigma$, Euler-Lagrange equations of the action $\int_{\Sigma}\Lambda$, where $\theta\vert_{\Sigma}= 0$, are given by:
\[
 \frac{\partial L}{\partial z}dx-\frac{d}{dx^\imath}\left(\frac{\partial L}{\partial p_\imath} \right)=0.
\]
\end{enumerate}
\end{rem}
\begin{defi}
 The unique form $\Pi:=\theta\wedge\Psi$ is called \textit{Poincaré-Cartan form}.
\end{defi}
\begin{thm}\label{thm1}

 A Monge-Ampère system $(\mathcal{M},\varepsilon=\{\theta,d\theta,\Psi\})$ is locally equivalent to an Euler-Lagrange system if and only if
\[
 d\Pi:=d(\theta\wedge\Psi)=\varphi\wedge\Pi,
\]
with $d\varphi\equiv0$ mod$\{\theta,d\theta,\Psi\}$.
\end{thm}
\begin{proof}
 See \cite{BRP} page 18-19.
\end{proof}
\begin{thm}\label{thm15}(Darboux)
 Given a 1-form $\theta$ as $\theta\wedge d\theta\neq0$. So it is written in the form
\[
 \theta=dz-p_1dx^1-p_2dx^2.
\]
\end{thm}

\section{Equivalence problem}

We apply the equivalence problem to classify the Monge-Ampère system $\varepsilon=(\mathcal{M},\theta,\Psi)$ satisfying
\[
 \left\{
\begin{array}{c}
\theta\wedge d\theta\wedge d\theta\neq0,\ \ \ \ \ \ \ \ \ \ \ \ \ \ \ \\
\Psi\wedge d\theta=0\text{ mod}(\theta),\ \ \ \ \ \ \ \ \ \ \ 
\end{array}
\right.
\]
by comparing it with another system $\tilde{\varepsilon}=(\tilde{\mathcal{M}}, \tilde{\theta},\tilde{\Psi})$, by looking at a diffeomorphism  $\varphi:\mathcal{M}\rightarrow\mathcal{\tilde{M}}$ such that 
\[
 \left\{
\begin{array}{c}
 \varphi^\ast\tilde{\theta}=\theta\ \ \\
\varphi^\ast\tilde{\Psi}=\Psi
\end{array}
\right.
\]

\subsection{Preliminaries} 

Let $\eta=\alpha\theta\neq0$, $\alpha\neq0$. Locally, by Darboux theorem, we can find 1-forms $\eta^0,\eta^1,\eta^2,\eta^3,\eta^4$ such that
\begin{equation}\label{e1}
 d\eta^0=\eta^1\wedge\eta^2+\eta^3\wedge\eta^4\ \ \text{ mod}(\theta),
\end{equation}
There exist functions $b_{\imath\jmath}$ such that $\Psi=\frac{1}{2}b_{\imath\jmath }\eta^\imath\wedge\eta^\jmath$, and as $\Psi\wedge d\eta^0=0$ mod$(\theta)$, then
\[
 b_{12}+b_{34}=0,
\]
We will study the conditions imposed by $\eta=(\eta^1,\eta^2,\eta^3,\eta^4)$ such that (\ref{e1}) be checked. So there are three non-zero orbits with we call: negative space, null space and positive space.
\begin{enumerate}
 \item  If $\Psi\wedge\Psi$ is a negative multiple of $d\eta^0\wedge d\eta^0$, then the local coframing $\eta$ may be chosen so that in addition of (\ref{e1}),
\[
 \Psi=\eta^1\wedge\eta^2-\eta^3\wedge\eta^4\ \ \text{mod}(\theta),
\]
for a classical variational problem, this occurs when the Euler-Lagrange PDE is hyperbolic.
\item  If $\Psi\wedge\Psi=0$, then $\eta$ may be chosen so that
\[
 \Psi=\eta^1\wedge\eta^3 \ \text{mod}(\theta),
\]
for a classical variational problem, this occurs when the Monge-Ampère PDE is parabolic.
\item  If $\Psi\wedge\Psi$ is a positive multiple of $d\eta^0\wedge d\eta^0$, then the local coframing $\eta$ may be chosen so that in addition of (\ref{e1}),
\[
 \Psi=\eta^1\wedge\eta^4-\eta^3\wedge\eta^2\ \ \text{mod}(\theta),
\]
for a classical variational problem, this occurs when the Monge-Ampère PDE is elliptic.
\end{enumerate}
In the following we will study the elliptic case, we look at the following conditions
\begin{equation}\label{equation50}
 \left\{
\begin{array}{c}
 \eta=\alpha\theta\neq0\ \ \ \ \ \ \ \ \ \ \ \ \ \ \ \ \ \ \ \ \ \ \ \ \ \ \ \\
d\eta^0=\eta^1\wedge\eta^2+\eta^3\wedge\eta^4\ \ \text{mod}(\theta),\\
\Psi=\eta^1\wedge\eta^4-\eta^3\wedge\eta^2\ \ \text{mod}(\theta),\ \
\end{array}
\right.
\end{equation}
\subsection{An algebra preliminary}
For $\omega:=(\omega^1,\omega^2,\omega^3,\omega^4)\in(\mathbb{R}^4)^\ast$, we consider the symmetric non-degenerate function 
\[
\langle.,.\rangle:\Lambda^2(\mathbb{R}^4)^\ast\times\Lambda^2(\mathbb{R}^4)^\ast\longrightarrow\ \ \ \ \mathbb{R},\ \ \ \ \ \ \ \ \ \ \ \ \ \ \ \ \ \ \ \ \ \ \ \ \ \ \ 
\]
\[
 \ \ \ \ \ \ \ \ \ \ \ \ \ \ \ \ \ \ \ \ \ \ \   ( \ \ \alpha\ \ \ \ ,\ \ \ \ \beta\ \ ) \ \longmapsto\ \  \langle\alpha,\beta\rangle:=\frac{\alpha\wedge\beta}{\omega^1\wedge\omega^2\wedge\omega^3\wedge\omega^4},
\]
The Lie algebra $SL(4,\mathbb{R})$ acts on $\Lambda^2(\mathbb{R}^4)^\ast\simeq\mathbb{R} ^6$ through the action $\forall$ $g\in SL(4,\mathbb{R})$
\[
q: \Lambda^2(\mathbb{R}^4)^\ast\ \ \longrightarrow\ \ \Lambda^2(\mathbb{R}^4)^\ast, \ \ \ \ \ \ \ \ \ \ \ \ \ \ \
\]
\[
 \alpha\ \ \ \ \longmapsto\ \ q(\alpha):= g^\ast\alpha,
\]
$q_g$ is a quadratic form in $g$ and we have
 \[
\langle q(\alpha),q(\beta)\rangle= \langle g^\ast\alpha,g^\ast\beta\rangle=\langle\alpha,\beta\rangle.
\]
 We want to represent $SL(4,\mathbb{R})=$Sp$(\mathbb{R}^{3,3}):=$Sp$(3,3)$\footnote{``Spin'' is a notation used by physicists. (Spin$(1,3)=SL(2,\mathbb{C}))$.} on $\mathbb{R}^6$. Let $G:=SO(\Lambda^2(\mathbb{R}^4)^\ast,\langle.,.\rangle)\subset GL(6,\mathbb{R})$. Denote
\[
 \Phi:\ \ SL(4,\mathbb{R})\ \ \ \longrightarrow\ \  G,\ \ \ \ \ \ \ \ \ \ \ \ \ \ \ \ \ \ \
\]
\[
 \alpha\ \ \ \ \ \longmapsto \ \ q(\alpha),
\]

 A basis $(\alpha^1_L,\alpha^2_L,\alpha^3_L,\alpha^1_R,\alpha^2_R,\alpha^3_R)$ of $\Lambda^2(\mathbb{R}^4)^\ast$ given by
\[
\left\{
 \begin{array}{c}
  \alpha^1_L=\omega^1\wedge\omega^2+\omega^3\wedge\omega^4,\\
\alpha^2_L=\omega^1\wedge\omega^3+\omega^4\wedge\omega^2,\\
\alpha^3_L=\omega^1\wedge\omega^4+\omega^2\wedge\omega^3,\\
\alpha^1_R=\omega^1\wedge\omega^2-\omega^3\wedge\omega^4,\\
\alpha^2_R=\omega^1\wedge\omega^3-\omega^4\wedge\omega^2,\\
\alpha^3_R=\omega^1\wedge\omega^4-\omega^2\wedge\omega^3,\\
 \end{array}
\right.
\]
We have $\forall a,b\in\{1,2,3\}$ 
\[
 \langle\alpha^a_L,\alpha^b_L\rangle=2\delta^a_b,\ \ \langle\alpha^a_R,\alpha^b_R\rangle=-2\delta^a_b \text{ and } \langle\alpha^a_L,\alpha^b_R\rangle=0,
\]
The signature of $\Phi$ is $(3,3)$ and $SO(\Lambda^2(\mathbb{R}^4)^\ast,\langle.,.\rangle)\subset SL(4,\mathbb{R})$ and since dim$SL(4,\mathbb{R})=$dim$SO(\Lambda^2(\mathbb{R}^4)^\ast,\langle.,.\rangle)=15$, we have
\[
 G=\text{Spin}(3,3),
\]
The first step of the equivalence problem method is to find a group $G_0$ and the associated $G_0-$structure $B_0$. Here $G_0$ preserves (\ref{equation50}), we define
\[
 G_0=G_{ellip}=\{g\in SL(4,\mathbb{R}), \ \ g^\ast\alpha^1_L=\alpha^1_L;\ \ g^\ast\alpha^3_L=\alpha^3_L\},
\]
For all $\xi$ in the Lie algebra $\frak{g}_{ellip}:=\frak{g}$, we have $\xi=(\xi)_{1\leq\imath,\jmath\leq4}\in M(4,\mathbb{R})$  tr$\xi=0$, and moreover, we have
\begin{equation}\label{equation31}
\left\{
\begin{array}{c}
 \xi\in M(4,\mathbb{R}), \ \ \text{tr}\xi=0,\\
\xi^\ast\alpha^1_L=\alpha^1_L;\ \ \xi^\ast\alpha^3_L=\alpha^3_L,
\end{array}
\right.
\end{equation}
Hence
\[
 \left\{
\begin{array}{c}
\xi^1_1+\xi^2_2+\xi^3_3+\xi^4_4=0,\ \ \ \ \ \ \ \ \ \ \ \ \ \ \ \ \ \ \ \ \ \ \ \ \ \ \ \ \ \ \ \ \ \ \ \ \ \ \ \ \ \ \ \ \ \ \ \ \ \ \ \ \ \ \ \ \ \ \ \ \ \ \\
 \xi^1_a\omega^a\wedge\omega^2+\xi^2_a\omega^1\wedge\omega^a+\xi^3_a\omega^a\wedge\omega^4+\xi^4_a\omega^3\wedge\omega^a=\omega^1\wedge\omega^2+\omega^3\wedge\omega^4,\\
\xi^1_a\omega^a\wedge\omega^4+\xi^4_a\omega^1\wedge\omega^a+\xi^2_a\omega^a\wedge\omega^3+\xi^3_a\omega^2\wedge\omega^a=\omega^1\wedge\omega^4+\omega^2\wedge\omega^3,\\
\end{array}
\right.
\]
Then there exist $a,b,c,d,e,f\in\mathbb{R}$ and a basis $\xi_1,\xi_2,\xi_3,\xi_4,\xi_5, \xi_6$ of $\frak{g}$ such as
\[
 \xi=a\xi_1+b\xi_2+c\xi_3+d\xi_4+e\xi_5+f\xi_6,
\]
with
\[
\xi_1=\left(\begin{array}{cccc}
             1&0&0&0\\
             0&-1&0&0\\
             0&0&1&0\\
             0&0&0&-1
            \end{array}
\right)\ \ \xi_2=\left(\begin{array}{cccc}
             0&0&0&0\\
             0&0&1&0\\
             0&0&0&0\\
             1&0&0&0
            \end{array}
\right)\ \ \xi_3=\left(\begin{array}{cccc}
             0&1&0&0\\
             0&0&0&0\\
             0&0&0&-1\\
             0&0&0&0
            \end{array}
\right)
\]
\[
\xi_4=\left(\begin{array}{cccc}
             0&0&1&0\\
             0&0&0&1\\
             -1&0&0&0\\
             0&-1&0&0
            \end{array}
\right)\ \ \xi_5=\left(\begin{array}{cccc}
             0&0&0&0\\
             -1&0&0&0\\
             0&0&0&0\\
             0&0&1&0
            \end{array}
\right)\ \ \xi_6=\left(\begin{array}{cccc}
            0&0&0&1\\
             0&0&0&0\\
             0&1&0&0\\
             0&0&0&0
            \end{array}
\right)
\]
Denote the $\mathbb{R}$-linear application $$\Phi:\ \ \frak{g}\ \ \longrightarrow \ \ sl(2,\mathbb{C}), \ \ \ \ \ \ \ \ \ \ \ \ \ \ \ \ \ \ \ \ \ \ \ \ \ $$
$$X,Y\longmapsto \Phi([X,Y])=[\Phi(X),\Phi(Y)].$$  
A basis of $sl(2,\mathbb{C})$ is $(h_0,e_0,f_0,h_1,e_1,f_1)$ such that
 $$[h_0,h_1]=[e_0,e_1]=[f_0,f_1]=0,\ \ [h_a,e_b]=-2i^{a+b}e_0,$$
$$ [h_a,f_b]=-2i^{a+b}f_0\text{ and }[e_a,f_b]=-i^{a+b}h_0,$$

We can find relations of the same type on the basis of $\frak{g}$.
\begin{table}[h!]
\begin{tabular}{|l|l|}
\hline
Structure constants of $\frak{g}$&Structure constants of $sl(2,\mathbb{C})$\\
\hline
 $[\xi_a,\xi_{3+a}]=0\ \ \text{ for } a=1,2,3$&$[h_0,h_1]=[e_0,e_1]=[f_0,f_1]=0$\\
\hline
$[\xi_1,\xi_2]=-2\xi_2,$&$[h_0,e_0]=-2e_0,$\\
\hline
$[\xi_1,\xi_3]=2\xi_3,$&$[h_0,f_1]=2f_1,$\\
\hline
$[\xi_1,\xi_5]=-2\xi_5,$&$[h_0,e_1]=-2e_1,$\\
\hline
$[\xi_1,\xi_6]=2\xi_6,$&$[h_0,f_0]=2f_0,$\\
\hline
$[\xi_4,\xi_2]=-2\xi_5,$&$[h_1,e_0]=-2e_1,$\\
\hline
$[\xi_4,\xi_6]=2\xi_3,$&$[h_1,f_0]=2f_1,$\\
\hline
$[\xi_4,\xi_5]=2\xi_2,$&$[h_1,e_1]=2e_0,$\\
\hline
$[\xi_4,\xi_3]=-2\xi_6,$&$[h_1,f_1]=-2f_0,$\\
\hline
$[\xi_2,\xi_6]=-\xi_1,$&$[e_0,f_0]=-h_0,$\\
\hline
$[\xi_2,\xi_3]=-\xi_4,$&$[e_0,f_1]=-h_1,$\\
\hline
$[\xi_5,\xi_6]=-\xi_4,$&$[e_1,f_0]=-h_1,$\\
\hline
$[\xi_5,\xi_3]=\xi_1,$&$[e_1,f_1]=h_0,$\\
\hline
\end{tabular}
\caption{ Comparison of the structure constants.}
\end{table}

We have this correspondence
$$\xi_1\longleftrightarrow h_0,$$
$$\xi_2\longleftrightarrow e_0,$$
$$\xi_3\longleftrightarrow f_1,$$
$$\xi_4\longleftrightarrow h_1,$$
$$\xi_5\longleftrightarrow e_1,$$
$$\xi_6\longleftrightarrow f_0,$$
Denote the linear map $$T:\ \ \ \mathbb{R}^4\ \ \ \longrightarrow\ \ \ \mathbb{C}^2\ \ \ \ \ \ \ \ \ \ $$
$$X=\left(\begin{array}{c}
           x^1\\
x^2\\
x^3\\
x^4
          \end{array}
\right)\longmapsto\left(\begin{array}{c}
           x^3+ix^1\\
x^2+ix^4\\
          \end{array}
\right),
$$
We can show that $\forall\ \ 1\leq\imath\leq 6$
$$T(\xi_\imath X)=\Phi(\xi_\imath)T(X),$$
\subsection{Back to the equivalence problem}

For $\omega=\left(\begin{array}{c}
\omega^0\\
\omega^1\\
\omega^2\\
\omega^3\\
\omega^4
          \end{array}
\right)$ and $\pi=\left(\begin{array}{c}
\pi^0\\
\pi^1\\
\pi^2\\
\bar{\pi}^1\\
\bar{\pi}^2
          \end{array}
\right)$ two vector valued 1-forms, $P\in M(5,\mathbb{C})$ we set $\omega=P\pi,$ where
\[
 \left\{
\begin{array}{c}
 \pi^0=\omega^0,\ \ \ \ \ \ \ \ \\
\pi^1=\omega^3+i\omega^1,\\
\pi^2=\omega^2+i\omega^4,\\
\bar{\pi}^1=\omega^3-i\omega^1,\\
\bar{\pi}^2=\omega^2-i\omega^4,\\
\end{array}
\right.
\]
Consider $\forall M_\natural=(a^\imath_\jmath)_{1\leq\imath,\jmath\leq4}\in\frak{g}$, we note $M\in M(5,\mathbb{C})$ by
$$M=\left(\begin{array}{cc}
                a^0_0&0\\
                 0&M_\natural
               \end{array}
\right),$$
Introducing the torsion $\tau$ given by $\tau=d\omega+\varphi\wedge\omega$, we obtain
\begin{equation}\label{equationb}
P^{-1}\tau=d\pi+\psi\wedge\pi,
\end{equation}
where $\psi=P^{-1}\varphi P$. Taking that $M_\natural=(a^\imath_\jmath) _{1\leq\imath,\jmath\leq4}\in\frak{g}$, thus
$$P^{-1}M P=\left(\begin{array}{ccccc}
              a^0_0&0&0&0&0\\
               0&a^1_1+ia^1_3&a^1_4+ia^1_2&0&0\\
               0&a^2_3-ia^2_1&a^2_2-ia^2_4&0&0\\
               0&0&0&a^3_3+ia^3_1&a^3_2+ia^3_4\\
                0&0&0&a^4_1-ia^4_3&a^4_4-ia^4_2\\
             \end{array}
\right),
$$ 
\begin{prop}\label{eqaution34}
 Let $(\mathcal{M}^5,\varepsilon)$ an elliptic Monge-Ampère system. The adapted coframings are the sections of $G_0-$structure on $\mathcal{M}$, where $G_0$ is the smallest subgroup generated by all matrices of size (1,2,2) of the form
\begin{equation}\label{equationa}
\left(\begin{array}{ccc}
              a^0_0&0&0\\
               C&A&0\\
               \bar{C}&0&\bar{A}\\
             \end{array}
\right),
\end{equation}
where $A\in$sl$(2,\mathbb{C})$ and det$A=a^0_0\neq0$.
\end{prop}
\begin{proof}
 The sections of $G-$structure adapted (\ref{equation50}) are of the form (\ref{equationa}).
\end{proof}
According to this propostion we can pass to the second step ``Calculation of the structure equations'', consider
\begin{equation}\label{equation32}
\psi=\left(\begin{array}{ccccc}
              \psi^0_0&0&0&0&0\\
               \psi^1_0&\psi^1_1&\psi^1_2&0&0\\
               \psi^2_0&\psi^2_1&\psi^2_2&0&0\\
               \bar{\psi}^1_0&0&0&\bar{\psi}^1_1&\bar{\psi}^1_2\\
               \bar{\psi}^2_0&0&0&\bar{\psi}^2_1&\bar{\psi}^2_2\\
             \end{array}
\right),
\end{equation} 
where $\psi^1_1+\psi^2_2=\bar{\psi}^1_1+\bar{\psi}^2_2=\psi^0_0$ .\\
We assume that $P^{-1}\tau=\left(\begin{array}{c}
                                 \tau^0\\ 
                                 \tau^1\\
                                  \tau^2\\
                                  \bar{\tau}^1\\
                                  \bar{\tau}^2\\
                                \end{array}
\right)$, where
$$
 \tau^0:=d\pi^0+\psi^0_0\wedge\pi^0=\frac{i}{2}(\bar{\pi}^1\wedge\bar{\pi}^2-\pi^1\wedge\pi^2),
$$
and for $\imath=1,2,3,4,$
$$
\tau^\imath=T^\imath_{12}\pi^1\wedge\pi^2+T^\imath_{1\bar{1}}\pi^1\wedge\bar{\pi}^1+T^\imath_{1\bar{2}}\pi^1\wedge\bar{\pi}^2+T^\imath_{2\bar{1}}\pi^2\wedge\bar{\pi}^1+T^\imath_{2\bar{2}}\pi^2\wedge\bar{\pi}^2+T^\imath_{\bar{1}\bar{2}}\bar{\pi}^1\wedge\bar{\pi}^2
$$
$$
+T^\imath_{01}\pi^0\wedge\pi^1+T^\imath_{02}\pi^0\wedge\pi^2+T^\imath_{0\bar{1}}\pi^0\wedge\bar{\pi}^1+T^\imath_{0\bar{2}}\pi^0\wedge\bar{\pi}^2.\ \ \ \ \ \ \ \ \ \ \ \ \ \ \ \ \ \ \ \
$$
 This produces the structure equations
\begin{equation}\label{equation36}
 \left\{
\begin{array}{c}
d\pi^0= -\psi^0_0\wedge\pi^0+\frac{i}{2}(\bar{\pi}^1\wedge\bar{\pi}^2-\pi^1\wedge\pi^2),\ \ \ \\
d\pi^1=-\psi^1_0\wedge\pi^0-\psi^1_1\wedge\pi^1-\psi^1_2\wedge\pi^2+\tau^1, \\
d\pi^2=-\psi^2_0\wedge\pi^0-\psi^2_1\wedge\pi^1-\psi^2_2\wedge\pi^2+\tau^2,\\
d\bar{\pi}^1=-\bar{\psi}^1_0\wedge\pi^0-\bar{\psi}^1_1\wedge\bar{\pi}^1-\bar{\psi}^1_2\wedge\bar{\pi}^2+\bar{\tau}^1,\\
d\bar{\pi}^2=-\bar{\psi}^2_0\wedge\pi^0-\bar{\psi}^2_1\wedge\bar{\pi}^1-\bar{\psi}^2_2\wedge\bar{\pi}^2+\bar{\tau}^2,\\
\end{array}
\right.
\end{equation}

Now we go to the second step which allows us to absorb the maximum of torsion in (\ref{equation36}) respecting $\psi^1_1+\psi^2_2=\bar{\psi}^1_1+\bar{\psi}^2_2=\psi^0_0$. First, by change the form $\psi^\imath_0\leftarrow\psi ^\imath_0-T^\imath_{0*}\pi^*$ we can consider\footnote{$*,\star \in\{1,2,\bar{1},\bar{2}\}.$}
$$T^\imath_{0*}=0.$$ 
By a change of $\psi^1_2$ and $\psi^2_1$, we can write
\[
 T^1_{2\bar{1}}=T^1_{2\bar{2}}=T^1_{12}=T^2_{1\bar{1}}=T^2_{1\bar{2}}=T^2_{12}=0,
\]
Respecting $\psi^1_1+\psi^2_2=\bar{\psi}^1_1+\bar{\psi}^2_2=\psi^0_0$, we can write
\[
 T^1_{1\bar{1}}=T^2_{2\bar{1}}=V_1\text{ and } T^2_{2\bar{2}}=T^1_{1\bar{2}}=V_2,
\]
Thus (\ref{equation36}) becomes
\begin{equation}\label{equationc}
 \left\{
\begin{array}{c}
d\pi^0= -\psi^0_0\wedge\pi^0+\frac{i}{2}(\bar{\pi}^1\wedge\bar{\pi}^2-\pi^1\wedge\pi^2),\ \ \ \ \ \ \ \ \ \ \ \ \ \ \ \ \ \ \ \ \ \ \ \ \ \ \ \ \ \ \ \ \ \ \ \ \ \ \ \ \ \ \ \ \ \\
d\pi^1=-\psi^1_0\wedge\pi^0-\psi^1_1\wedge\pi^1-\psi^1_2\wedge\pi^2+V_1\pi^1\wedge\bar{\pi}^1+V_2\pi^1\wedge\bar{\pi}^2+U_1\bar{\pi}^1\wedge\bar{\pi}^2, \\
d\pi^2=-\psi^2_0\wedge\pi^0-\psi^2_1\wedge\pi^1-\psi^2_2\wedge\pi^2+V_1\pi^2\wedge\bar{\pi}^1+V_2\pi^2\wedge\bar{\pi}^2+U_2\bar{\pi}^1\wedge\bar{\pi}^2,\\
d\bar{\pi}^1=-\bar{\psi}^1_0\wedge\pi^0-\bar{\psi}^1_1\wedge\bar{\pi}^1-\bar{\psi}^1_2\wedge\bar{\pi}^2+\bar{V}_1\bar{\pi}^1\wedge\pi^1+\bar{V}_2\bar{\pi}^1\wedge\pi^2+\bar{U}_1\pi^1\wedge\pi^2,\\
d\bar{\pi}^2=-\bar{\psi}^2_0\wedge\pi^0-\bar{\psi}^2_1\wedge\bar{\pi}^1-\bar{\psi}^2_2\wedge\bar{\pi}^2+\bar{V}_1\bar{\pi}^2\wedge\pi^1+\bar{V}_2\bar{\pi}^2\wedge\pi^2+\bar{U}_2\pi^1\wedge\pi^2,\\
\end{array}
\right.
\end{equation}
Here $V_\imath$ and $U_\imath$ are the new coefficients of torsion which are expressed in terms of $T^\imath_{*\star}$.\\
After calculating $0\equiv d(d\pi^0)$, we have
$$U_1=-2\bar{V}_2,\ \ U_2=2\bar{V}_1.$$ 
We calculate $d(d\pi^1)\equiv0$ and $d(d\pi^2)\equiv0$, thus we have the relation mod $\{\pi^0,\pi^1,\pi^2\}$.

 \begin{equation}\label{equation39}
0\equiv d\left(
\begin{array}{c}
U_1\\
U_2
\end{array}
\right)+\frac{i}{2}\left(
\begin{array}{c}
\psi^1_0\\
\psi^2_0
\end{array}
\right)+\left(
\begin{array}{cc}
\psi^1_1&\psi^1_2\\
\psi^2_1&\psi^2_2
\end{array}
\right).\left(
\begin{array}{c}
U_1\\
U_2
\end{array}
\right)-\psi^0_0\left(
\begin{array}{c}
U_1\\
U_2
\end{array}
\right),
\end{equation}

Let $G_1$-structure $B_1\subset B_0$ in which $\tau^1=\tau^2=0$ and $\varphi^1_0$, $\varphi^2_0$ are semi-basic, we consider the projector $\Phi: B_0\rightarrow B_1$ such that, for $x\in B_0$ we associate $x.g_0$ is a submersion which respects fibers. Thus $G_1$ is a sub-group acting over $B_1$ generated by matrices of the form
\begin{equation}\label{equation42}
 g_1=\left(
\begin{array}{ccc}
 a&0&0\\
0&A&0\\
0&0&\bar{A}\\
\end{array}
\right),
\end{equation}
Denote by
\[
 \psi^\imath_0=P^\imath_0\pi^0+P^\imath_*\pi^*, \text{ and } \bar{\psi}^\imath_0=\bar{P}^\imath_0\pi^0+\bar{P}^\imath_*\pi^*
\]
So the structure equations read
\begin{equation}\label{equation41}
\left\{
\begin{array}{c}
d\pi^0= -\psi^0_0\wedge\pi^0+\frac{i}{2}(\bar{\pi}^1\wedge\bar{\pi}^2-\pi^1\wedge\pi^2), \\
d\pi^1=-\psi^1_1\wedge\pi^1-\psi^1_2\wedge\pi^2-P^1_*\pi^*\wedge\pi^0, \\
d\pi^2=-\psi^2_1\wedge\pi^1-\psi^2_2\wedge\pi^2+P^2_*\pi^*\wedge\pi^0,\\
\end{array}
\right.
\end{equation}
We absorb the torsion, respecting the condition $\psi^1_1+\psi^2_2=\psi^0_0$
\begin{equation}\label{equation23}
\left\{
\begin{array}{c}
d\pi^0= -\psi^0_0\wedge\pi^0+\frac{i}{2}(\bar{\pi}^1\wedge\bar{\pi}^2-\pi^1\wedge\pi^2),\ \ \ \ \ \ \ \ \ \ \ \  \ \ \ \ \ \ \ \ \ \ \ \ \  \ \ \ \ \ \ \ \\
d\pi^1=-\psi^1_1\wedge\pi^1-\psi^1_2\wedge\pi^2-P{\pi}^1\wedge\pi^0-P^1_{\bar{1}}\bar{\pi}^1\wedge\pi^0-P^1_{\bar{2}}\bar{\pi}^2\wedge\pi^0, \\
d\pi^2=-\psi^2_1\wedge\pi^1-\psi^2_2\wedge\pi^2-P{\pi}^2\wedge\pi^0-P^2_{\bar{1}}\bar{\pi}^2\wedge\pi^0-P^2_{\bar{2}}\bar{\pi}^2\wedge\pi^0,\\
\end{array}
\right.
\end{equation}
Respecting $\psi^1_1+\psi^2_2=\bar{\psi}^1_1+\bar{\psi}^2_2=\psi^0_0$, we  can show
\begin{equation}\label{ei}
P+\bar{P}=0.
\end{equation}
We have $$0=-d\psi^0_0\wedge\pi^0+\frac{i}{2}\psi^0_0\wedge\bar{\pi}^1\wedge\bar{\pi}^2-\frac{i}{2}\psi^0_0\wedge\pi^1\wedge\pi^2$$
$$
+\frac{i}{2}d\bar{\pi}^1\wedge\bar{\pi}^2-\frac{i}{2}\bar{\pi}^1\wedge d\bar{\pi}^2-\frac{i}{2}d\pi^1\wedge\pi^2+\frac{i}{2}\pi^1\wedge d\pi^2,$$
thus
$$2id\psi^0_0\wedge\pi^0=(2\bar{P}\bar{\pi}^1\wedge\bar{\pi}^2-2P\pi^1\wedge\pi^2-(P^2_{\bar{1}}+\bar{P}^2_{\bar{1}})\pi^1\wedge\bar{\pi}^1+(\bar{P}^1_1-P^2_{\bar{2}})\pi^1\wedge\bar{\pi}^2$$
$$+(P^1_{\bar{1}}-\bar{P}^2_{\bar{2}})\pi^2\wedge\bar{\pi}^1+(\bar{P}^1_{\bar{2}}-P^1_{\bar{2}})\pi^2\wedge\bar{\pi}^2)\wedge\pi^0,   \ \ \ \ \ \ \ \ \ $$
thus
$$P-\bar{P}=0,$$
for (\ref{ei}), then we have
$$P=0,$$
then (\ref{equation23}) reads
\begin{equation}\label{equation24}
\left\{
\begin{array}{c}
d\pi^0= -\psi^0_0\wedge\pi^0+\frac{i}{2}(\bar{\pi}^1\wedge\bar{\pi}^2-\pi^1\wedge\pi^2),\ \ \ \ \ \ \ \ \ \ \ \  \ \ \ \ \ \\
d\pi^1=-\psi^1_1\wedge\pi^1-\psi^1_2\wedge\pi^2-P^1_{\bar{1}}\bar{\pi}^1\wedge\pi^0-P^1_{\bar{2}}\bar{\pi}^2\wedge\pi^0, \\
d\pi^2=-\psi^2_1\wedge\pi^1-\psi^2_2\wedge\pi^2-P^2_{\bar{1}}\bar{\pi}^2\wedge\pi^0-P^2_{\bar{2}}\bar{\pi}^2\wedge\pi^0,\\
\end{array}
\right.
\end{equation}
in particular
$$
2id\psi^0_0=-(P^2_{\bar{1}}+\bar{P}^2_{\bar{1}})\pi^1\wedge\bar{\pi}^1+(\bar{P}^1_1-P^2_{\bar{2}})\pi^1\wedge\bar{\pi}^2$$
\begin{equation}\label{equa24}
+(P^1_{\bar{1}}-\bar{P}^2_{\bar{2}})\pi^2\wedge\bar{\pi}^1+(\bar{P}^1_{\bar{2}}-P^1_{\bar{2}})\pi^2\wedge\bar{\pi}^2.
\end{equation}
 We define a pair of $2\times2$ matrix-valued functions on $B_1$ by
$$
S_1=\left(
\begin{array}{cc}
P^1_{\bar{1}}+\bar{P}^2_{\bar{2}}&\bar{P}^1_{\bar{2}}+P^1_{\bar{2}}\\
P^2_{\bar{1}}-\bar{P}^2_{\bar{1}}&\bar{P}^1_1+P^2_{\bar{2}}
\end{array}
\right), \ \
S_2=\left(
\begin{array}{cc}
P^1_{\bar{1}}-\bar{P}^2_{\bar{2}}&\bar{P}^1_{\bar{2}}-P^1_{\bar{2}}\\
P^2_{\bar{1}}+\bar{P}^2_{\bar{1}}&\bar{P}^1_1-P^2_{\bar{2}}
\end{array}
\right).
$$

\begin{thm}
 An elliptic Monge-Ampère system $(\mathcal{M},\varepsilon)$ satisfies $S_1=S_2=0$ if and only if it is locally equivalent to the Monge-Ampère system for the linear homogeneous Laplace equations.
\end{thm}
\begin{proof}
 If $S_2=0$, then $\psi^0_0$ is closed, from  (\ref{equa24}), $S_2=0$ if and only if for some 1-form $\alpha$ we have
$$
d\psi^0_0=\alpha\wedge\pi^0.
$$
But $dd=0$, hence $0\equiv-\alpha\wedge d\pi^0$, which gives us
$$\alpha\equiv0\ \ \text{mo}\{ \pi^0 \}$$
Conversely, if $d\psi^0_0=0$, then $S_2=0$.
In case $S_1=S_2=0$, then $d\psi^0_0=0$, thus we can locally find a function $\lambda>0$ such that
$$\psi^0_0=\lambda^{-1}d\lambda$$
In case $S_1=S_2=0$ we can find
$$d(\pi^1\wedge\pi^2)=-\psi^0_0\pi^1\wedge\pi^2$$
hence, we can write
$$d(\lambda\omega^1\wedge\omega^4)=d(\lambda\omega^3\wedge\omega^2)=d(\lambda\omega^1\wedge\omega^2)=d(\lambda\omega^3\wedge\omega^4)=0$$
Then locally by (\ref{thm15}) there exist a functions $x,y,p$ and $q$ such that
$$-dp\wedge dx=\lambda\omega^1\wedge\omega^2$$
$$-dq\wedge dy=\lambda\omega^3\wedge\omega^4$$
$$-dp\wedge dy=\lambda\omega^1\wedge\omega^4$$
$$-dq\wedge dx=\lambda\omega^3\wedge\omega^2$$
Not that
$$d(\lambda\pi^0)=d(\lambda\omega^0)=\frac{i}{2}(\bar{\pi}^1\wedge\bar{\pi}^2-\pi^1\wedge\pi^2)=\lambda(\omega^1\wedge\omega^2+\omega^3\wedge\omega^4)=-dp\wedge dx-dq\wedge dy$$
By Poincaré lemma, locally there is exist a function $z$, such that
$$\lambda\omega^0=dz-pdx-qdy$$
Then, in local coordinates, our elliptic Monge-Ampère system is
$$
\varepsilon=\{\omega^0,\omega^1\wedge\omega^2+\omega^3\wedge\omega^4,\omega^1\wedge\omega^4-\omega^3\wedge\omega^2\}\ \ \ \ \ \ \ \ \ \ \ \  \ \ \ \ \ \ \ \ \ \ \
$$
$$
=\{dz-pdx-qdy,-dp\wedge dx-dq\wedge dy,-dp\wedge dy+dq\wedge dx\}.
$$
\end{proof}

It's natural to ask about the situation in wich $S_2=0$, but possibly $S_1\neq0$.

\begin{thm}
 An elliptic Monge-Ampère system $(\mathcal{M},\varepsilon)$ satisfies $S_2=0$ if and only if it is locally equivalent to an Euler-Lagrange system.
\end{thm}
\begin{proof}
 The condition for $\varepsilon$ to contain a Poincaré-Cartan form 
 $$
 \Pi=\frac{1}{2}\lambda\pi^0\wedge(\bar{\pi}^1\wedge\bar{\pi}^2+\pi^1\wedge\pi^2).
 $$
 $$
 =\lambda\omega^0\wedge(\omega^1\wedge\omega^4-\omega^3\wedge\omega^2).
 $$
 We can assume that $\Pi$ to be closed for some $\lambda>0$ on $B_1$. By differentiating then
 $$
 0=(d\lambda-2\lambda\psi^0_0)\wedge\omega^0\wedge(\omega^1\wedge\omega^4-\omega^3\wedge\omega^2).
 $$
 Exterior algebra, for some function $\mu$, say
 $$
 d\lambda-2\lambda\psi^0_0=\mu\lambda\omega^0.
 $$
 In other words,
 $$
 d(\log\lambda)-2\psi^0_0=\mu\omega^0.
 $$
 Hence
 $$
 d\psi^0_0\equiv0\text{ mod }\{\omega^0\}.
 $$
But we know that
 $$
 d\omega^0=\omega^1\wedge\omega^2+\omega^3\wedge\omega^4.
 $$
  (\ref{equa24}), gives us $S_2=0$.
 
\end{proof}

\subsection{Remark in Cartan's test}

\begin{defi}
 If $(\pi^0,\pi^1,\pi^2)$ be a lifted coframe, then the associated Exterior Differential System, with equivalence condition $\pi^0\wedge\pi^1\wedge\pi^2\wedge\bar{\pi}^1\wedge\bar{\pi}^2\neq0$, it is involutive if and only if it satisfies the Cartan's test.
\end{defi}
To apply equivalence method to some problem there are several steps, one important is Cartan's test. If the problem is involutive we can conclude, if this is not the case,
it is necessary to extend the system to continue. We begin, for example, to test the involution in the elliptic case. To find this there is a process to follow 
\cite{Olver}, in (\ref{equation24}), we have $r=5$ and $n=3$; to find the reduced characters of Cartan, replace in (\ref{equation24}) $\psi^\imath_\jmath$ by
$z^\imath_{\jmath0}\pi^0+z^\imath_{\jmath1}\pi^1+ z^\imath_{\jmath2}\pi^2$, we can show
\begin{equation}\label{equation25}
\left\{
\begin{array}{c}
 z^0_{0\jmath}=0 \ \ \ \jmath=0,...,2,\ \ \ \  \\
z^1_{10}=0, \ \ z^1_{11}= z^2_{12},\ \ \ \ \ \\
z^1_{20}=0, \ \ z^1_{12}= z^1_{21}, z^1_{22}, \\
z^2_{10}=0, \ \ z^2_{11},\ \ \ \ \ \ \ \ \ \ \ \ \ \\
\end{array}
\right.
\end{equation}

The four parameters $z^1_{11},z^1_{12},z^1_{22},z^2_{11}$ can be chosen arbitrarily, thus the degree of indeterminancy $r^{(1)}$ of a lifted coframe is the number 
of free variables in the solution to the associaled linear absorption system
\[
 r^{(1)}=4.
\]
Let be $X=(x^0,...,x^2)\in\mathbb{R}^3$ and the matrix $M$ of size $3\times 4$ define by
\[
 M(X):=M^{\imath l}_k(X):=\sum_{\jmath=0}^2A^{\imath l}_{\jmath k}x^\jmath,\ \ \imath=0,...,2, \ \ (^l_k)\in(^0_0,^1_1,^1_2,^2_1),
\]
where $A^{\imath l}_{\jmath k}$ are a coefficients define in (\ref{equation24}). In other words
\[
 M(X)=\left(A^{\imath l}_{0k}(x^0)+A^{\imath l}_{1k}(x^1)+A^{\imath l}_{2k}(x^2)\right)_{\substack{0\leq\imath\leq2 \\ (^l_k)\in(^0_0,^1_1,^1_2,^2_1)}},
 \]
  Thus 
 \[
 M(X)=\left(
\begin{array}{ccccccc}
 -x^0&0&0&0&\\
0&-x^1&-x^2&0\\
-x^2&x^2&0&-x^1\\
\end{array}
\right).
\]
For $X=(-1,-1,0)$, thus
\[
 M=\left(
\begin{array}{ccccccc}
 1&0&0&0\\
0&1&0&0\\
0&0&0&1\\
\end{array}
\right).
\]
 If we denote by $s'_1,...,s'_3$ the reduced characters Cartan then
\[
 s'_1=3.
\]
Now, for $X=(x^0,...,x^2)$ and $Y=(y^0,...,y^2)$, we have
\[
\left(
\begin{array}{c}
 M(X)\\
 M(Y)
 \end{array}
 \right)
 =\left(
\begin{array}{ccccccc}
 -x^0&0&0&0\\
0&-x^1&-x^2&0\\
-x^2&x^2&0&-x^1\\
-y^0&0&0&0\\
0&-y^1&-y^2&0\\
-y^2&y^2&0&-y^1\\
\end{array}
\right).
\]
For $X=(-1,-1,0)$ and $Y=(0,0,-1)$, we have $s'_1+s'_2=4$, thus 
\[
s'_2=1.
\]
Or we have $s'_1+s'_2+s'_3=r=4$, then
\[
s'_3=0,
\]
thus we have
\[
s'_1+2s'_2+3s'_3=5> r^{(1)}=4.
\]
Hence the system (\ref{equation24}) is not satisfies Cartan's test, thus it's necessary to extend the system to continue. Note that before the step of normalizing, the system (\ref{eqaution34}) satisfied Cartan's test, I gave a proof of this in my thesis \cite{Imsatfia}. This leads to further investigations.

\textbf{Acknowledgement} I am grateful to Frédéric Hélein who suggested to me this problem as a part of my Ph.D.

\nocite{*}
\bibliographystyle{plain}
\bibliography{article3}

\begin{thebibliography}{1}

\bibitem{Aubin2001}
Thierry Aubin.
\newblock {\em A course in differential geometry}, volume~27 of {\em Graduate
  Studies in Mathematics}.
\newblock American Mathematical Society, Providence, RI, 2001.

\bibitem{BRP}
Robert Bryant, Phillip Griffiths, and Daniel Grossman.
\newblock {\em Exterior differential systems and {E}uler-{L}agrange partial
  differential equations}.
\newblock Chicago Lectures in Mathematics. University of Chicago Press,
  Chicago, IL, 2003.

\bibitem{Helein2004}
Fr{\'e}d{\'e}ric H{\'e}lein.
\newblock Hamiltonian formalisms for multidimensional calculus of variations
  and perturbation theory.
\newblock In {\em Non compact problems at the intersection of geometry,
  analysis, and topology}, volume 350 of {\em Contemp. Math.}, pages 127--147.
  Amer. Math. Soc., Providence, RI, 2004.

\bibitem{Thomas}
Thomas~A. Ivey and J.~M. Landsberg.
\newblock {\em Cartan for beginners: differential geometry via moving frames
  and exterior differential systems}, volume~61 of {\em Graduate Studies in
  Mathematics}.
\newblock American Mathematical Society, Providence, RI, 2003.

\bibitem{Imsatfia}
Imsatfia Moheddine.
\newblock {\em Géométrie de Cartan fondée sur la notion d'aire et application
  du problème d'équivalence}.
\newblock PhD thesis, 2012.

\bibitem{Motimoto79}
Tohru Morimoto.
\newblock La géométrie des équations de {M}onge-{A}mpère.
\newblock {\em C. R. Acad. Sci. Paris Sér. A-B}, 289(1):A25--A28, 1979.

\bibitem{Olver}
Peter~J. Olver.
\newblock {\em Equivalence, invariants, and symmetry}.
\newblock Cambridge University Press, Cambridge, 1995.

\bibitem{Spivak79}
Michael Spivak.
\newblock {\em A comprehensive introduction to differential geometry. {V}ol.
  {(I-II-IV)}}.
\newblock Publish or Perish Inc., Wilmington, Del., second edition, 1979.

\bibitem{Neut}
Neut Sylvain.
\newblock {\em Implantation et nouvelles applications de la méthode
  d'équivalence de Cartan}.
\newblock PhD thesis, 2003.

\end{thebibliography}

\end{document}